\documentclass[12pt]{article}
\usepackage{amsfonts}
\usepackage{amsmath}
\usepackage{amsthm}

\newtheorem{theorem}{Theorem}[section]
\newtheorem{lemma}[theorem]{Lemma}

\title{On Coloring the Odd-Distance Graph}

\author{Jacob Steinhardt \\
\small \texttt{jsteinha@mit.edu}}

\date{August 10, 2009 }

\begin{document}

\maketitle

\begin{abstract}
We present a proof, using spectral techniques, that there is no finite measurable coloring of the odd-distance graph.
\end{abstract}

\section{Introduction}

Let $\mathcal{O}$ be the graph with $V(\mathcal{O}) = \mathbb{R}^2$ and where two vertices are connected if they are at an odd distance from each other. We call $\mathcal{O}$ the \textit{odd-distance graph}. Let the \emph{measurable chromatic number} of a graph denote the least number of colors needed to color a graph such that each color class is measurable. We aim to show that the measurable chromatic number $\chi$ of $\mathcal{O}$ is infinite. We will do this by defining a sequence of operators $B_\alpha$ related to the adjacency operator of $\mathcal{O}$. We then use an extension of the well-known spectral inequality, $\chi(G) \geq 1-\frac{\lambda_{\max}}{\lambda_{\min}}$, to the infinite-dimensional case. We next determine the set of eigenfunctions for $B_{\alpha}$ (they turn out to be the characters of $\mathbb{R}^2$, though this is in a sense guaranteed from the Fourier analysis on $\mathbb{R}^2$). This gives us the full set of eigenvalues for the $B_\alpha$, which we then bound below in order to show that $1-\frac{\lambda_{\max}}{\lambda_{\min}}$ goes to $\infty$ as $\alpha$ goes to $1$. As this is a lower bound for $\chi(\mathcal{O})$, we will then have established that $\mathcal{O}$ has infinite measurable chromatic number. Throughout the paper, whenever we refer to chromatic number we will always mean measurable chromatic number.

This result has been proven elsewhere, for example as a consequence of the theorem that all measurable sets with positive density at infinity contain all sufficiently large distances (see e.g. \cite{Fal}). However, the proof in this paper uses primarily techniques from spectral graph theory rather than measure theory and seems to be closer in spirit to the problem itself. See also \cite{Bac} where a similar generalization of the Lovasz theta function is studied and used to derive new bounds on the chromatic number of unit-distance graphs in $\mathbb{R}^n$.

Section $2$ generalizes Hoffman's eigenvalue bound (see e.g. \cite{Lov}) to the case of a family of weighted adjacency matrices for $\mathcal{O}$. This family is paramterized by a real number $\alpha$, $\alpha < 1$. Section $3$ then shows that this gives a bound of $\Omega((\alpha-1)^{\frac{-3}{4}})$ on $\chi$, so in particular the bound goes to $\infty$ as $\alpha$ goes to $1$. This implies that $\chi(\mathcal{O})$ is infinite. Section $4$ consists of concluding remarks and possible ideas for generalizing this technique to deal with non-measurable colorings.

\section{A Generalization of Hoffman's Bound}

Consider the operator $B_{\alpha} : L^2(\mathbb{R}^2) \to L^2(\mathbb{R}^2)$ defined by

\begin{equation}
\label{bdef1}
(B_{\alpha}f)(x,y) = \int_{-\pi}^{\pi} \sum_{k=0}^{\infty} \alpha^{-k} f(x+(2k+1)\cos(\theta),y+(2k+1)\sin(\theta)) d\theta
\end{equation}

Clearly, $B_{\alpha}$ is a linear operator. We also make the following observation:

\begin{lemma}
\label{zprod1}
Let $I$ be an independent set in $\mathcal{O}$, and let $g$ be any function that is zero outside of $I$. Then $\langle f,B_{\alpha}f\rangle = 0$. 
\end{lemma}

\begin{proof}

\begin{eqnarray*}
\label{zprod2}
\langle f,B_{\alpha}f\rangle & = & \iint_{\mathbb{R}^2} f(x,y)\overline{(B_{\alpha}f)(x,y)} dA\\
 & = & \iint_{\mathbb{R}^2} f(x,y) \int_{-\pi}^{\pi} \sum_{k=0}^{\infty} \alpha^{-k} \overline{f(x+(2k+1)\cos(\theta),y+(2k+1)\sin(\theta))} d\theta dA\\
 & = & \iint_{\mathbb{R}^2} \int_{-\pi}^{\pi} \sum_{k=0}^{\infty} \alpha^{-k} f(x,y)\overline{f(x+(2k+1)\cos(\theta),y+(2k+1)\sin(\theta))} d\theta dA\\
 & = & 0
\end{eqnarray*}

In the last equality we used the fact that 

\[
f(x,y)\overline{f(x+(2k+1)\cos(\theta),y+(2k+1)\sin(\theta))} = 0
\]

since not both $(x,y)$ and $(x+(2k+1)\cos(\theta),y+(2k+1)\sin(\theta))$ can be in $I$ (they are at odd distance), so $f$ applied to at least one of the two must be zero.
\end{proof}

We can use this to bound the chromatic number $\chi$ of $\mathcal{O}$. Let $C_\alpha = I-\frac{\alpha-1}{2\pi}B_\alpha$, where $I$ is the identity. Then $C_\alpha$ is equivalent to convolution by some function, and so is diagonalized by the Fourier transform on $\mathbb{R}^2$. Therefore, its operator norm is equal to its largest eigenvalue. We thus have the following:

\begin{lemma}
\begin{equation}
\label{bound1}
\chi \geq \frac{\rho(C_\alpha)}{\rho(C_\alpha)-1}
\end{equation}
\end{lemma}

\begin{proof}
By the preceeding comments, it suffices to show that $\chi \geq \frac{||A||}{||A||-1}$. Suppose that there exists a $\chi$-coloring of $\mathcal{O}$ with color classes $I_1,\ldots,I_{\chi}$.  Let $S_r$ be a circle with radius $r$ centered at the origin. Let $f_i$ be defined as

\begin{equation}
f_i(x) = \left\{
\begin{array}{lr}
1 & x \in I_i \cap S_r\\
0 & x \not\in I_i \cap S_r
\end{array} \right\}
\end{equation}

Let $f = f_1+\ldots+f_\chi$. We note that each $f_i$ satisfies the conditions of Lemma~\ref{zprod1}. Therefore, $\langle f_i, C_\alpha f_i \rangle = \langle f_i, f_i \rangle$. We then have:

\begin{eqnarray*}
2(\chi-1)||A||||f||^2 & = & \sum_{i,j=1}^{\chi} ||A||||f_i-f_j||^2\\
 & \geq & \sum_{i,j=1}^{\chi} \langle f_i-f_j, C_{\alpha}(f_i-f_j) \rangle\\
 & = & \sum_{i,j=1}^{\chi} \langle f_i,C_{\alpha}f_i \rangle + \langle f_j,C_{\alpha}f_j \rangle - \langle f_i,C_{\alpha}f_j \rangle - \langle f_j,C_{\alpha}f_i \rangle\\
 & = & \left(\sum_{i,j=1}^{\chi} ||f_i||^2+||f_j||^2\right)-2\sum_{i,j=1}^{\chi} \langle f_i,C_{\alpha}f_j \rangle\\
 & = & 2\chi||f||^2-2\langle\sum_{i=1}^{\chi} f_i,C_{\alpha}(\sum_{j=1}^{\chi} f_j\rangle\\
 & = & 2\chi||f||^2-2\langle f,C_{\alpha}f \rangle
\end{eqnarray*}

So $2(\chi-1)||A||||f||^2 \geq 2\chi||f||^2-2\langle f,C_{\alpha}f \rangle$. This re-arranges to $\chi(||A||||f||^2-||f||^2) \geq ||A||||f||^2-\langle f, C_{\alpha}f \rangle$, or $\chi \geq \frac{||A||}{||A||-1}\left(1-\frac{\langle f, C_{\alpha} f \rangle}{||f||^2}\right)$. We will bound $\langle f, C_{\alpha} f \rangle$ in terms of $\alpha$ and $r$. Let $r = 2k+1$, where $r$ is an integer. Let $D_{\alpha}=I-C_{\alpha}$. Then it suffices to show show that $\frac{\langle f, D_{\alpha} f \rangle}{||f||^2}$ approaches $1$ as $k \to \infty$. For a point at distance between $2j$ and $2j+2$ from the origin, $(D_{\alpha} f)(x,y) \geq (1-\alpha) \left(1+\alpha+\ldots+\alpha^{k-j}\right)$ for $j = 0,\ldots,k-1$. Therefore, $\langle f, D_{\alpha} f \rangle$ is bounded below by the sum 

\begin{equation}
\sum_{j=0}^{k-1} (1-\alpha^{k+1-j}) \pi ((2j+2)^2-(2j)^2)
\end{equation}

This simplifies to $\pi \left((2k)^2-8\frac{\alpha^{k+2}-(k+1)\alpha^2+k\alpha}{(\alpha-1)^2}+4\frac{\alpha^{k+1}-\alpha}{\alpha-1}\right)$. On the other hand, $||f||^2 = \pi(2k+1)^2$, so we want to look at the quantity

\begin{equation}
\frac{(2k)^2-8\frac{\alpha^{k+2}-(k+1)\alpha^2+k\alpha}{(\alpha-1)^2}+4\frac{\alpha^{k+1}-\alpha}{\alpha-1}}{(2k+1)^2}
\end{equation}

We break this up into the two quantities

\begin{equation}
\frac{(2k)^2}{(2k+1)^2}-4\frac{\alpha^{k+2}+\alpha^{k+1}-(2k+1)\alpha^2+2k\alpha-\alpha}{(2k+1)^2(\alpha-1)^2}
\end{equation}

Clearly $\frac{(2k)^2}{(2k+1)^2}$ tends to $1$ as $k \to \infty$. If the numerator of the other quantity does not tend to $\infty$, then we are done since the denominator does tend to $\infty$. Otherwise, we can use L'hopital's rule, from which we get that the second quantity tends to

\begin{equation}
4\frac{\alpha^{k+2}\ln(\alpha)+\alpha^{k+1}\ln{\alpha}-2\alpha^2+2\alpha}{(8k+4)(\alpha-1)^2}
\end{equation}

The top is clearly bounded as $k \to \infty$ (remember $\alpha < 1$), and the bottom is clearly unbounded, so this expression goes to $0$ as $k \to \infty$, so that $\frac{\langle f, D_{\alpha} f}{||f||^2}$ does indeed tend to $1$. Therefore, we can let $r \to \infty$, so that $\frac{\langle f, C_{\alpha} f \rangle}{||f||^2} \to 0$, and we get the desired bound.
\end{proof}

\section{Using the Spectral Bound}

We next compute the eigenvalues of $B_{\alpha}$ (if $\lambda$ is an eigenvalue of $B_{\alpha}$, then $1-\frac{\alpha-1}{2\pi}\lambda$ is an eigenvalue of $C_{\alpha}$). Since $B_\alpha$ is diagonalized by the Fourier transform, $f_{(r,s)}(x,y) = e^{i(rx+sy)}$ with $r,s \in \mathbb{R}$ are the eigenfunctions of $B_{\alpha}$. We see that the eigenvalue of the eigenfunction $f_{(r,s)}$ is given by

\begin{equation}
\label{eigval1}
\lambda_{(r,s)} = \int_{-\pi}^{\pi} \sum_{k=0}^{\infty} \alpha^{-k} e^{i(2k+1)(r\cos(\theta)+s\sin(\theta))} d\theta = \int_{-\pi}^{\pi} \sum_{k=0}^{\infty} \alpha^{-k} e^{i(2k+1)\sqrt{r^2+s^2}\cos(\theta+\phi)} d\theta
\end{equation}

for an appropriately chosen $\phi$. Thus we need only actually consider $\lambda_{(r,0)}$, which we from now on denote $\lambda(r)$. Then we have

\begin{equation}
\label{eigval2}
\lambda(r) = \int_{-\pi}^{\pi} \sum_{k=0}^{\infty} \alpha^{-k} \left(e^{ir\cos(\theta)}\right)^{2k+1} = \int_{-\pi}^{\pi} \frac{e^{ir\cos(\theta)}}{1-\alpha^{-1}e^{2ir\cos(\theta)}} d\theta
\end{equation}

Here we have simply summed the geometric series. Since $B_{\alpha}$ is symmetric, $\lambda(r)$ must be real. Letting $x = r\cos(\theta)$, we can take the real part of the integral:

\begin{eqnarray*}
\label{eigval3}
\lambda(r) & = & Re\left[\int_{-\pi}^{\pi} \frac{(\cos(x)+i\sin(x))(1-\alpha^{-1}\cos(2x)+i\alpha^{-1}\sin(2x))}{(1-\alpha^{-1}\cos(2x))^2+\alpha^{-2}\sin(2x)^2} d\theta\right]\\
 & = & \int_{-\pi}^{\pi} \frac{\cos(x)(1-\alpha^{-1}\cos(2x))-\alpha^{-1}\sin(x)\sin(2x)}{1+\alpha^{-2}-2\alpha^{-1}\cos(2x)} d\theta\\
 & = & \int_{-\pi}^{\pi} \alpha \frac{\alpha \cos(x)-\cos(x)\cos(2x)-\sin(x)\sin(2x)}{\alpha^2+1-2\alpha\cos(2x)} d\theta\\
 & = & \int_{-\pi}^{\pi} \alpha \frac{\alpha \cos(x)-\cos(x)}{\alpha^2+1-2\alpha\cos(2x)} d\theta \\
 & = & \int_{-\pi}^{\pi} \frac{\alpha(\alpha-1)\cos(x)}{(\alpha-1)^2+4\alpha\sin^2(x)} d\theta \\
 & = & \int_{-\pi}^{\pi} \frac{\alpha(\alpha-1)\cos(r\cos(\theta))}{(\alpha-1)^2+4\alpha\sin^2(r\cos(\theta))} d\theta
\end{eqnarray*}

In the second-to-last step, we used the identity $\cos(a-b)=\cos(a)\cos(b)+\sin(a)\sin(b)$. We will show that the magnitude of $\lambda_{\min}$ is at most $O((\alpha-1)^{-\frac{3}{4}})$, which shows that $\rho(C_\alpha) = 1+O((\alpha-1)^{\frac{1}{4}})$. This will show that as $\alpha$ approaches $1$, $\frac{\rho(C_\alpha)}{\rho(C_\alpha)-1}$ grows without bound, so that there cannot exist any finite coloring of $\mathcal{O}$.

Note that for $r \leq \frac{\pi}{2}$, $\lambda(r)$ is necessarily positive since the integrand is always positive ($\cos(r\cos(\theta))$ being the only thing that can go negative in the expression). We thus assume that $r > \frac{\pi}{2}$. It suffices to show that

\begin{equation}
\int_{0}^{\frac{\pi}{2}} \frac{(\alpha-1)\cos(r\cos(\theta))}{(\alpha-1)^2+4\alpha\sin^2(r\cos(\theta))} d\theta \geq -c(\alpha-1)^{-\frac{3}{4}}-d
\end{equation}

for all $r$ for some constants $c,d$ (as this, neglecting a factor of $4\alpha$, is clearly an upper bound for the integral above). Let $h$ be the function we are integrating. Let $\mathcal{R}_k$ denote the region for which $|h(\theta)| \geq 1$ and that contains the value of $\theta$ where $\cos(\theta) = \frac{k\pi}{r}$. Then we note that $\displaystyle|\int_{\mathcal{R}_k} h(x) dx| > |\int_{\mathcal{R}_{k-1}} h(x) dx|$ since $\cos(\theta)$ decreases faster as $\theta$ increases from $0$ to $\frac{\pi}{2}$. Also, the signs of these integrals alternate, so we can either throw out all of them or all but the first one, depending on whether the integral of $h$ across $\displaystyle\mathcal{R}_{\lfloor\frac{r}{\pi}\rfloor}$ is positive or negative. If it is positive, then we have thrown out all of the integral, except for a part where $|h(x) < 1|$, so that the remaining part of the integral is obviously bounded. Thus we will assume that the integral of $h$ across $\displaystyle\mathcal{R}_{\lfloor\frac{r}{\pi}\rfloor}$ is negative. We will bound the area of $\displaystyle\mathcal{R}_{\lfloor\frac{r}{\pi}\rfloor}$. First, we determine when

\begin{equation}
\label{bound2}
\frac{\alpha-1}{(\alpha-1)^2+4\alpha\sin^2(r\cos(\theta))} \geq 1
\end{equation}

as this is clearly a superset of the area where $h(\theta) \geq 1$. But this happens when $\alpha-1 \geq (\alpha-1)^2+4\alpha\sin^2(r\cos(\theta))$, or $\sin^2(r\cos(\theta)) \leq \frac{(\alpha-1)-(\alpha-1)^2}{4\alpha} = (\alpha-1)\frac{2-\alpha}{4\alpha} < \frac{\alpha-1}{4}$. So the area for which (\ref{bound2}) holds is contained in the area for which $\sin(r\cos(\alpha)) \in [-\frac{\sqrt{\alpha-1}}{2},\frac{\sqrt{\alpha-1}}{2}]$. On the other hand, this is contained in the area in which $r\cos(\theta)$ is within $\sqrt{\frac{\alpha-1}{2}}$ of a multiple of $\pi$, as $\sin(\sqrt{\frac{\alpha-1}{2}}) > \sqrt{\frac{\alpha-1}{2}}-\frac{(\alpha-1)^{1.5}}{12\sqrt{2}} > \frac{\sqrt{\alpha-1}}{2}$ for $\alpha-1$ small enough. So we want to find when

\begin{equation}
\label{bound3}
-\frac{1}{r}\sqrt{\frac{\alpha-1}{2}} \leq \frac{k\pi}{r}-\cos(\theta) \leq \frac{1}{r}\sqrt{\frac{\alpha-1}{2}}
\end{equation}

We claim that, if $\cos(\theta_0) = \frac{k\pi}{r}$, then it suffices to take $\theta \in [\theta_0-\frac{2\sqrt[4]{\alpha-1}}{\sqrt{r}},\theta_0+\frac{2\sqrt[4]{\alpha-1}}{\sqrt{r}}]$. First of all, if $\theta_0-\frac{\sqrt{\alpha-1}}{r} < 0$ or $\theta_0+\frac{\sqrt{\alpha-1}{r}} > \frac{\pi}{2}$, then $\theta$ is outside of our range of integration and so we are definitely covering at least the area we need on that end of the interval. Thus we may assume otherwise, and we have the following lemma:

\begin{lemma}
\label{bound4}
If $d > 0$ and $\theta,\theta+d \in [0,\frac{\pi}{2}]$, then $\cos(\theta)-\cos(\theta+d) \geq 1-\cos(d)$.
\end{lemma}

\begin{proof}
Take $\frac{d}{d\theta} \left[\cos(\theta)-\cos(\theta+d)\right] = \sin(\theta+d)-\sin(\theta)$. This is clearly increasing for $\theta \in [0,\frac{\pi}{2}-d]$, so we might as well take $\theta = 0$, as this gives a smaller value for $\cos(\theta)-\cos(\theta+d)$ than any legal value of $\theta$. Then we get $1-\cos(d)$ as our answer, as claimed.
\end{proof}

With Lemma~\ref{bound4} in hand, we need only show that $1-\cos(\frac{2\sqrt[4]{\alpha-1}}{\sqrt{r}}) > \frac{1}{r}\sqrt{\frac{\alpha-1}{2}}$. This is evident once again from the Taylor approximation as, for $\alpha-1$ small enough, $1-\cos(\frac{2\sqrt[4]{\alpha-1}}{\sqrt{r}}) > \frac{2\sqrt{\alpha-1}}{r}-\frac{2(\alpha-1)}{3r^2} > \frac{1}{r}\sqrt{\frac{\alpha-1}{2}}$. Thus for any given value of $k$, the area for which (\ref{bound2}) holds is at most $\frac{4\sqrt[4]{\alpha-1}}{\sqrt{r}}$. We only care about $\displaystyle\mathcal{R}_{\lfloor\frac{r}{\pi}\rfloor}$, so in particular we can take $k = \lfloor\frac{r}{\pi}\rfloor$ and the preceding argument holds. On the other hand, $\frac{\alpha-1}{(\alpha-1)^2+4\alpha\sin^2(r\cos(\theta))} < \frac{1}{\alpha-1}$, so integrating across this entire region gives us a value whose magnitude is at most $\frac{4}{\sqrt{r}(\alpha-1)^{\frac{3}{4}}}$. Integrating across the rest of the interval $[0,\frac{\pi}{2}]$ gives us a value of magnitude at most $\frac{\pi}{2}$, since we have shown that the integral across all of the remaining $\displaystyle\mathcal{R}_k$, $k < \lfloor\frac{r}{\pi}\rfloor$, must yield a positive number, and for all other portions of the interval $|h(\theta)| < 1$ by design. Also, recall that we established that $r > \frac{\pi}{2}$, so in particular $r > 1$. Thus we have that

\begin{equation}
\int_{0}^{\frac{\pi}{2}} \frac{\alpha-1}{(\alpha-1)^2+4\alpha\sin^2(r\cos(\theta))} d\theta \geq -4(\alpha-1)^{-\frac{3}{4}}-\frac{\pi}{2}
\end{equation}

as desired. This establishes that the measurable chromatic number of the odd-distance graph is infinite.

\section{Conclusion and Open Problems}

The largest remaining question is whether or not the chromatic number in the normal sense (without requiring measurability) is infinite or not. Perhaps the first thing to ask is how reliable a spectral bound is for talking about non-measurable colorings. There is a famous example of a graph in which the chromatic number depends upon the axioms of set theory adopted -- in particular, upon adopting choice versus determinacy (which states that all subsets of $\mathbb{R}^n$ are Lebesgue measurable). It is the graph with vertex set the real line where to vertices are connected if their distance is $\sqrt{2}+q$ for some rational $q$. 

We can color the connected component of $0$ with only two colors by coloring $n\sqrt{2}+q$ based on the parity of $n$. As all other connected components are translates of this one, we can then color the entire graph by taking a representative from each component and translating the coloring by that representative (this is where we use choice). However, no measurable coloring of this graph exists with even countably many colors. For a more detailed description, see the original paper by Shelah and Soifer \cite{She}.

Interestingly, an attempt to use Hoffman's bound on this graph only gives a lower bound of $2$ if we try the same strategy of weighting the edges so that the weighted degree of each vertex is finite, then letting the edge weights all tend to $1$. In a sense this is similar to considering only finite subgraphs (a la Erd\"{o}s-deBruijn), as the contribution of all but finitely many of the edges from each vertex is vanishingly small. It is still somewhat stronger, though, as we still consider all vertices while only caring about some of the edges connected to each vertex.

It would be interesting to find a graph in which there exists a non-measurable coloring smaller than that given by Hoffman's bound (or, even better, to find conditions under which Hoffman's bound is valid for all colorings). As noted in the preceding paragraph, the most obvious candidate for such a graph fails.

Finally, we consider possible ways of improving the result presented in this paper to the non-measurable case. Lovasz, in his initial paper (\cite{Lov}) on the $\vartheta$ function, gives many alternate characterizations of the Lovasz theta function, which is essentially what we are using here. It seems plausible that one of them could be made more amenable to dealing with colorings of infinite graphs. For example, if we can assign to each vertex $x$ a vector $\vec{v}_x$ such that $\langle \vec{v}_x, \vec{v}_y \rangle = 0$ whenever $x$ and $y$ are non-adjacent, then for any $\vec{c}$ the chromatic number is bounded above by $\sum_{x} |\langle \vec{v}_x, \vec{c} \rangle|^2$. This itself does not lend itself well to the case of uncountably many vertices, but for countably many vertices it seems much more plausible that it can be used to say something about the chromatic number. Thus we may make some progress by studying sublattices of $\mathcal{O}$. The author has tried this for certain sublattices, but has so far been unsuccessful.

One lattice that seems somewhat promising is the triangular tiling of the plane, that is, points of the form $a(1,0)+b(\frac{1}{2},\frac{\sqrt{3}}{2})$ for $a,b \in \mathbb{Z}$. In particular, it would be interesting if we could show using spectral techniques that the chromatic number was greater than $5$. I have tried to do this but have so far been unable to show that the appropriate generalization of the $\vartheta$ function for this graph takes on a value greater than $4$.

\section{Acknowledgements}

The author would like to thank Matt Kahle, Jonathan Kelner and Daniel Kane for helpful conversations.


\end{document}